\pdfoutput=1

\documentclass[12pt,a4paper]{amsart}
\usepackage{amsfonts}
\usepackage{wasysym}
\usepackage{amsthm}
\usepackage{amsmath}
\usepackage{amscd}

\usepackage[latin2]{inputenc}
\usepackage{t1enc}
\usepackage[mathscr]{eucal}
\usepackage{indentfirst}
\usepackage{graphicx}
\usepackage{graphics}
\usepackage{subcaption}
\usepackage{pict2e}
\usepackage{epic}
\numberwithin{equation}{section}
\usepackage[margin=2.9cm]{geometry}
\usepackage{epstopdf}
\usepackage{hyperref}

\theoremstyle{plain}
\newtheorem{Th}{Theorem}[section]
\newtheorem{Lemma}[Th]{Lemma}
\newtheorem{Cor}[Th]{Corollary}
\newtheorem{Prop}[Th]{Proposition}

\theoremstyle{definition}
\newtheorem{Def}[Th]{Definition}

\newtheorem{Rem}[Th]{Remark}
\newtheorem{?}[Th]{Problem}

\newtheorem*{theorem*}{Condition A}

\begin{document}
	
	\title{alternating links with totally geodesic checkerboard surfaces}
	
	\author[Hong-Chuan Gan]{Hong-Chuan Gan}
	
	\address{Department of Mathematics, Sun Yat-sen University, 510275, Guangzhou, China}
	
	\email{hcganmath@gmail.com}
	
	
	
	\begin{abstract} 
		We prove that alternating links with two totally geodesic checkerboard surfaces are three links with projection the 1-skeleton of the octahedron, the cuboctahedron and the icosidodecahedron.
		Then we characterize these links as right-angled completely realisable links and show that all hyperbolic weaving knots with two exceptions have both checkerboard surfaces not totally geodesic.
	\end{abstract}
	
	\maketitle
	
	\section{Introduction}
	A knot or link is hyperbolic if its complement in $S^3$ admits a complete hyperbolic metric.
	This metric is unique up to isometry by Mostow-Prasad rigidity\cite{mostow2016strong,prasad1973strong}.
	Due to works of Menasco, it is known that there exists a complete hyperbolic metric on the complement of a prime alternating link in $S^3$ that is not a (2,$q$)-torus link\cite{menasco1984closed}.
	An embedded or immersed surface in the knot complement is called totally geodesic if it is isotopic to a surface that lifts to a set of geodesic planes in $\mathbb{H}^3$.
	Adams and Schoenfeld utilize lifts of rigid hyperbolic 2-orbifolds to generate totally geodesic Seifert surfaces and give some examples of emmbedded totally geodesic surfaces in knot and link complement\cite{adams2005totally}, for example the balanced Pretzel links has one totally geodesic checkerboard surface.
	In \cite{adams2008totally}, the authors generalize these results and give more examples.
	In particular, they prove that alternating links created from the 1-skeleton of the octahedron, the cuboctahedron and the icosidodecahedron have two totally geodesic checkerboard surfaces.
	Champanerkar, Kofman and Purcell provide examples of two links living in $\mathbb{T}^2 \times I$ where $\mathbb{T}^2$ is the torus, namely the triaxial link and the square-weave link, such that both checkerboard surfaces on the torus are totally geodesic\cite{champanerkar2019geometry}.
	It is an open question whether any alternating knot admits two totally geodesic checkerboard surfaces\cite{champanerkar2019geometry}.
	More generally, one can ask whether any knot has a projection with both checkerboard surfaces totally geodesic\cite{adams2016spanning}.
	
	In this article, we show that the only alternating links in $S^3$ with two totally geodesic checkerboard surfaces are alternating links with projection the 1-skeleton of the octahedron,the cuboctahedron and the icosidodecahedron and characterize them as right-angled completely realisable links.
	Completely realisable alternating links are alternating links whose checkerboard polyhedra can be realized directly as ideal hyperbolic polyhedra which can then be glued together to give the complete hyperbolic structure on the link complement.
	Aitchison and Reeves show that if the checkerboard polyhedron of a completely realisable alternating links is realized as a 3-valent polyhedron, then the polyhedron has the combinatorial type of a prism or a 3-valent Archimedean solid\cite{aitchison2002archimedean}.
	There is also completely realisable alternating links of which checkerboard polyhedron can be realized as a 4-valent ideal hyperbolic polyhedron, for example, weaving knots $W(3,n)(n\geq3)$(\cite[Example 6.8.11]{thurston1979geometry},\cite[Example 7.4]{aitchison2002archimedean}).
	We show that $L$ is a completely realisable alternating links of which checkerboard polyhedra can be realized as 4-valent right-angled ideal hyperbolic polyhedra iff they are one of three links mentioned above.
	A right-angled hyperbolic polyhedron is a hyperbolic polyhedron with dihedral angles $\pi/2$.
	In \cite{champanerkar2019right}, Champanerkar, Kofman and Purcell associate a set of hyperbolic right-angled polyhedra to any reduced,prime,alternating link diagram of a link $L$ and prove that volume sum of this set called right-angled volume(denoted by $vol^{\perp}(L)$) is a geometric link invariant.
	They ask \textit{Does there exist a hyperbolic alternating link $L$, besides the Borromean link, for which $vol^{\perp}(L) = vol(L)$?}.
	We show that three links mentioned above are such links.
	We also resolve a case of their conjecture on right-angled knot; see Remark \ref{resolveCKP}.%
	
	\subsection{Organization}
	In Section \ref{definitions}, we recall some definitions in hyperbolic knot theory.
	In Section \ref{proof of the main theorem}, we prove that if checkerboard surfaces of an alternating knot or link $L$ are totally geodesic, then they intersect each other at right angles.
	If in addition n-gons in the diagram are regular, then $L$ is one of three links.
	Then we characterize these three links as right-angled completely realisable links and show that the regularity condition can be removed.
	
	\subsection{Acknowledgement}
	I thank Jessica Purcell for helpful discussions.
	I thank the anonymous referees for many helpful suggestions, especially for strengthening Proposition \ref{regular imply no} to Theorem \ref{main main results}.
	
	\section{definitions} \label{definitions}
	\begin{Def}
		An \textit{alternating} diagram is a diagram of a link with an orientation such that when following a component of the link in the direction of the orientation, the crossings alternate between over and under along the component.
		A knot or link is called \textit{alternating} if it admits a alternating diagram.
	\end{Def}
	\begin{Def}
		A link $L \subset S^3$ with at least two components is a \textit{split} link if there is a 2-sphere in $S^3-L$ separating $S^3$ into two balls and each ball contains at least one component of $L$.
		A link diagram $D$ in $S^2$ is a \textit{split} diagram if there is a simple closed curve $\gamma$ in the projection plane $S^2$ and $\gamma$ separates $S^2$ into two discs each containing part of $D$.
	\end{Def}
	\begin{Def}
		A \textit{reducible crossing} is a crossing through which we may draw a circle $\gamma$ on the projection plane such that $\gamma$ meets the diagram only at the crossing.
		A diagram is \textit{reduced} if it contains no reducible crossings.
	\end{Def}
	\begin{Def}
		A \textit{crossing arc} is defined to be an embedded arc in $S^3$ with endpoints on the knot or link $K$, which projects to a single point lying at a crossing in the diagram of $K$.	
	\end{Def}
	\begin{Def}
		A projection of the knot or link diagram $D$ divides the projection plane $S^2$ into regions.
		A region with n crossings on its boundary will be called a \textit{projection $n$-gon} of the diagram $D$.
	\end{Def}
	\begin{Def}
		An \textit{$n$-gon} coresponding to a projection $n$-gon of the knot or link diagram $D$ is a disk in the complement with boundary alternating between knot and crossing arc(see Figure \ref{checkerboard locally}).
		In the following we abuse terminology and say \textit{an $n$-gon in D} or \textit{D has an $n$-gon}.
		We say an $n$-gon is \textit{opposite} to an $m$-gon if their projections share a common vertex but do not share any edge.
	\end{Def}
	\begin{Def}
		Let $K$ be a knot or link.
		Consider the checkerboard coloring of the projection of diagram $D$ of $K$.
		The \textit{checkerboard surface} is constructed by gluing $n$-gons corresponding to projection $n$-gons in the checkerboard coloring along crossing arcs.
		See Figure \ref{checkerboard}.
	\end{Def}
	Let $K$ be a hyperbolic alternating knot or link given by a reduced alternating diagram $D$, then checkerboard surfaces of $D$ are essential\cite{menasco1993classification,purcell2020hyper} and quasifuchsian(\cite[Theorem 1.9]{adams2007noncompact},\cite{futer2014quasifuchsian}) in ${S}^3-K$.
	Crossing arcs of a reduced alternating diagram of a hyperbolic alternating knot or link lift and are homotopic to geodesics in $\mathbb{H}^3$\cite{thistlethwaite2014alternative}.
	If one checkerboard surface $S$ is totally geodesic, then $S$ is isotopic to $S'$ which lifts to totally geodesic planes in $\mathbb{H}^3$ and crossing arcs of the reduced alternating diagram lift and are isotopic to geodesics in these totally geodesic planes.
	\begin{figure}
		\centering
		\captionsetup[subfigure]{width=0.95\linewidth, justification=raggedright}%
		\begin{subfigure}[t]{.3\textwidth}
			\centering
			\includegraphics[width=.8\linewidth]{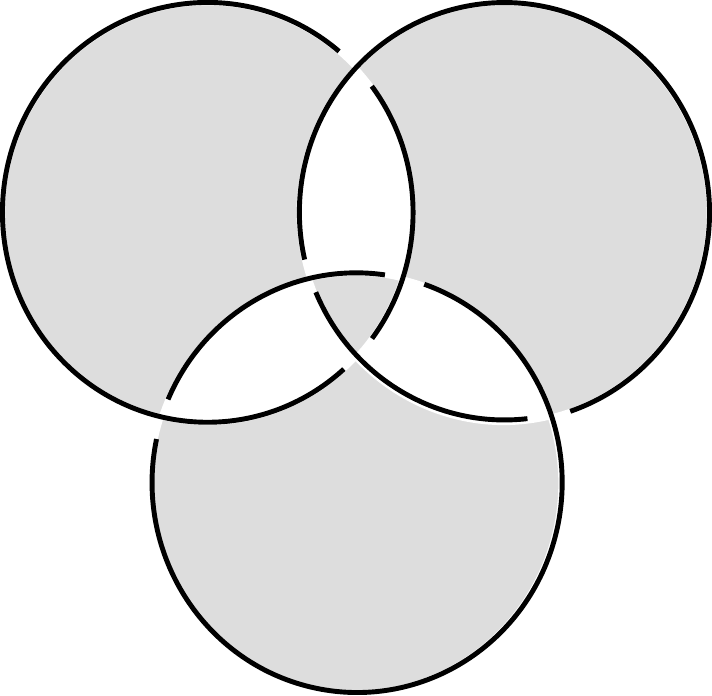}
			\caption{a checkerboard surface of the Borromean rings}
			\label{checkerboard of Borromean}
		\end{subfigure}%
		\begin{subfigure}[t]{.5\textwidth}
			\centering
			\includegraphics[width=1.2\linewidth]{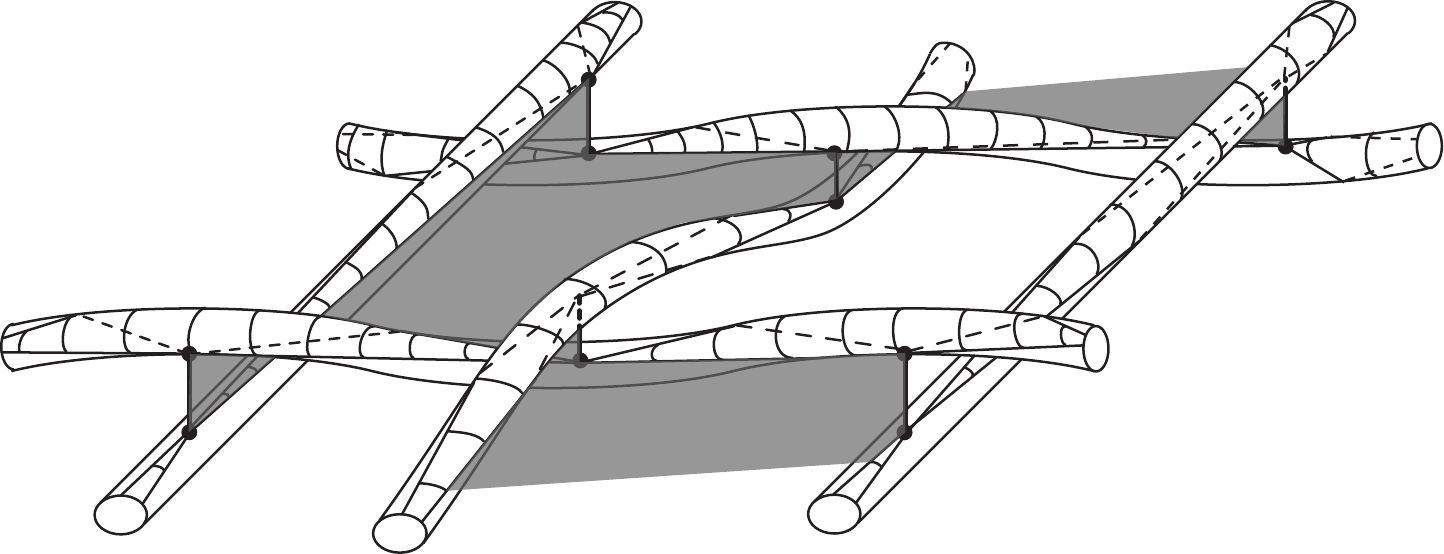}
			\caption{checkerboard surface near crossing arcs; figure modified from \protect\cite{adams2006cusp}}
			\label{checkerboard locally}
		\end{subfigure}%
	\caption{}
	\label{checkerboard}
	\end{figure}

	\subsection{checkerboard decomposition}See \cite{menasco20polyhedra,purcell2020hyper} and \cite[Chapter 2]{weeks1985hyperbolic} for more details.
	Let $L$ be a hyperbolic alternating knot or link, let $D$ be its reduced alternating diagram.
	Cutting the complement $S^3-L$ along two checkerboard surfaces of $D$ obtains two (topological) polyhedra.
	After collopsing the edges coming from the link, we obtain two combinatorial polyhedra with 4-valent vertices and a checkerboard coloring coming from the checkerboard coloring of the projection graph.
	They are mirror image of each other.
	Removing their vertices obtains ideal polyhedra.
	Each one is called a \textit{checkerboard polyhedron} associated to the link diagram $D$.
	One can glue them to obtain the link complement $S^3-L$.
	The gluing rotates faces with one coloring by one edge in the clockwise direction, and rotates faces with the other coloring by one edge in the counterclockwise direction.
	These two topological polyhedra do not necessarily agree with the complete hyperbolic structure.
	That is, there may not be two ideal hyperbolic polyhedra that have the same combinatorial type as checkerboard polyhedra and can be glued in the same manner as the checkerboard polyhedra to give the complete hyperbolic structure of $S^3-L$.
	We define horoball neighborhood of a link following \cite{futer2016survey}.
	\begin{Def}
		The hyperbolic knot or link complement $M$ admits a unique complete hyperbolic structure by Mostow-Prasad rigidity.
		The ends of M have the form $\mathbb{T}^2 \times [1,\infty)$.
		Under the covering map $p:\mathbb{H}^3 \to M$, each end is geometrically realized as the image of a horoball $H_i \subset \mathbb{H}^3$.
		The preimage of each end is a collection of horoballs.
		Shrinking $H_i$ if necessary, we can ensure that all horoballs in the preimage of an end have disjoint interiors in $\mathbb{H}^3$.
		For such a choice of $H_i$, $p(H_i)=C_i$ is said to be a horoball neighborhood of the cusp $C_i$, or \textit{horocusp} in $M$.
	\end{Def}

\section{proofs of the main results}	\label{proof of the main theorem}
\subsection{Totally geodesic checkerboard surfaces are perpendicular}
	The following lemma is a result in \cite{adams2008totally}
	\begin{Lemma} \label{no bigon region}
		Let L be a hyperbolic alternating knot or link with a reduced alternating diagram $D$.
		If checkerboard surfaces of D are totally geodesic in $S^3-L$, then D has no 2-gon.
	\end{Lemma}
	Note that since $D$ is a diagram on the projection plane(a sphere in $S^3$), the outermost $n$-gon is considered too.
	\begin{proof}
		If not, suppose $G_2$ is a 2-gon in $D$.
		Let $S$ be the checkerboard surface that does not contain $G_2$.
		Since $S$ is totally geodesic, by \cite[Theorem 3.2]{adams2008totally}, $G_2$ is an essential 2-gon in the complement of $S$.
		By \cite[Theorem 3.1]{adams2008totally}, $S$ cannot be totally geodesic.
		This contradicts to the assumptions.
	\end{proof}
	
	\begin{Lemma} \label{3-gon exists}
		Let L be an alternating knot or link with a reduced alternating diagram $D$.
		If the projection graph G of D on the projection plane $S^2$ has no 2-gons, then G has a 3-gon.
	\end{Lemma}
	\begin{proof}
		We use Euler's formula $V-E+F=2$, where $V$,$E$ and $F$ denote the number of vertices,edges and faces of G respectively, to lead to a contradiction.
		Since $G$ is a finite 4-valent graph on  $S^2$, $E = V*4/2$.
		Since $D$ is reduced, $G$ has no 1-gons.
		If $G$ has no 2-gons and 3-gons, then $ F \leq E*2/4 $. Hence
		$$ V-E+F \leq E/2-E+E/2 = 0 < 2 $$
		which is a contradiction.
	\end{proof}
	As a result of Lemma \ref{no bigon region} and Lemma \ref{3-gon exists}, if $L$ is a hyperbolic alternating knot or link with a reduced alternating diagram $D$ and two checkerboard surfaces of $D$ are totally geodesic, then $D$ has a 3-gon.
	\begin{Th} \label{right angle}
		Let K be a hyperbolic alternating knot or link.Let D be its reduced alternating diagram.If two checkerboard surfaces of D are totally geodesic in $S^{3}-K$, then they intersect each other at right angles.
	\end{Th}
	\begin{proof}
		If $K$ is a knot, let $N(K)$ be a horocusp of the knot complement $S^3-K$.
		Denote two checkerboard surfaces of $D$ by $S_1$ and $S_2$.
		Consider a horosphere $H$ in the preimage of $\partial N(K)$ in $\mathbb{H}^3$.
		Intersections of $H$ and the lifts of $S_i$ are parallel lines on $H$ since $S_i$ is embedded in the knot complement for $i={1,2}$.
		Therefore $S_1$ and $S_2$ intersect $H$ in a quadrangulation pattern; see Figure $\ref{quadragulation of horosphere}$.
		By Lemma \ref{no bigon region} and Lemma \ref{3-gon exists}, there is a checkerboard surface say $S_{1}$ of which defining region contains a projection 3-gon.
		The corresponding 3-gon in $S_1$ lifts and is isotopic to a totally geodesic ideal triangle $\Delta$ in $\mathbb{H}^{3}$ and lifts of $S_{2}$ intersects $\Delta$ at all its three edges(which are isotopic to lifts of three crossing arcs).
		Use the upper half space model, put a vertex of $\Delta$ at $\infty$ and take a bird view.
		$\Delta$ is the red segment in Figure \ref{explain}.
		Two blue lines are lifts of $S_2$ which are parallel since $S_2$ is embedded in the link complement.
		But lifts of $S_2$ also intersect $\Delta$ at the third edge.
		Hence there is a blue circle passing two endpoints of the red segment.
		If the intersection angles of $\Delta$ and $S_2$ are not right angles, this circle intersects two blue lines, contradicting the fact that $S_2$ is embedded.
		The only possibility is Figure \ref{explain2}.
		\begin{figure} 
			\begin{subfigure}{0.4\textwidth}
				\centering
				\includegraphics[width=.8\linewidth]{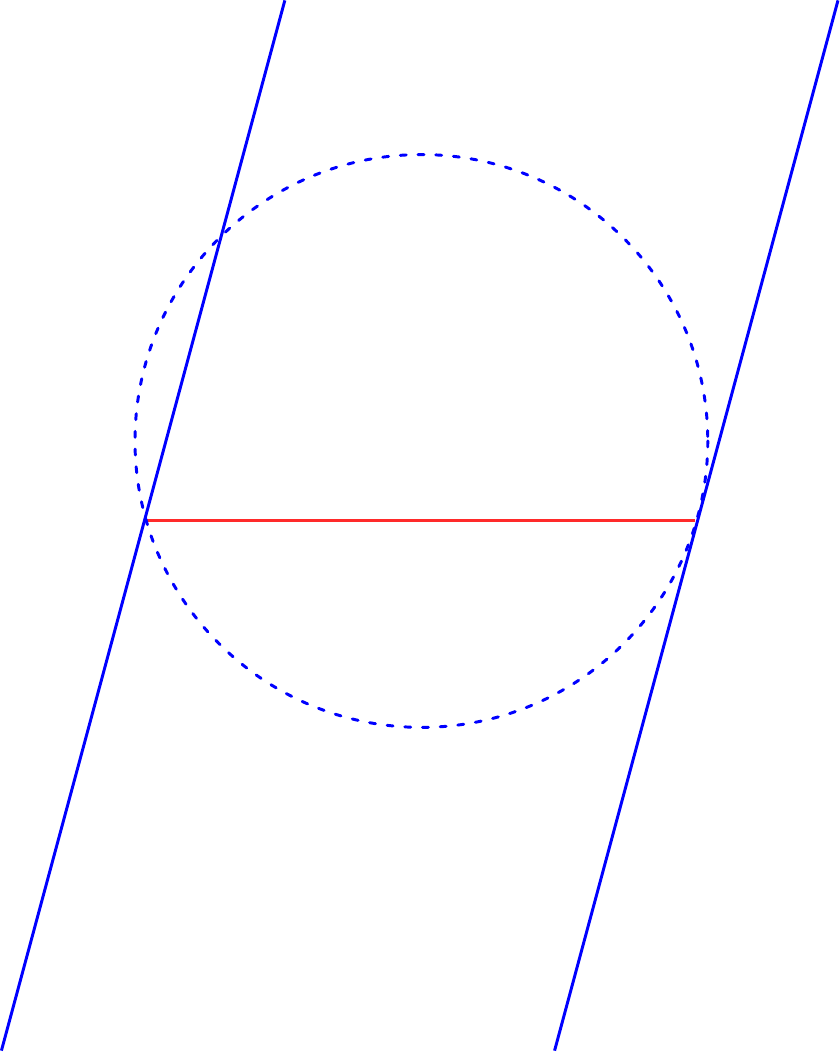}
				\caption{}
				\label{explain}
			\end{subfigure}%
			\begin{subfigure}{0.4\textwidth}
				\centering
				\includegraphics[width=.5\linewidth]{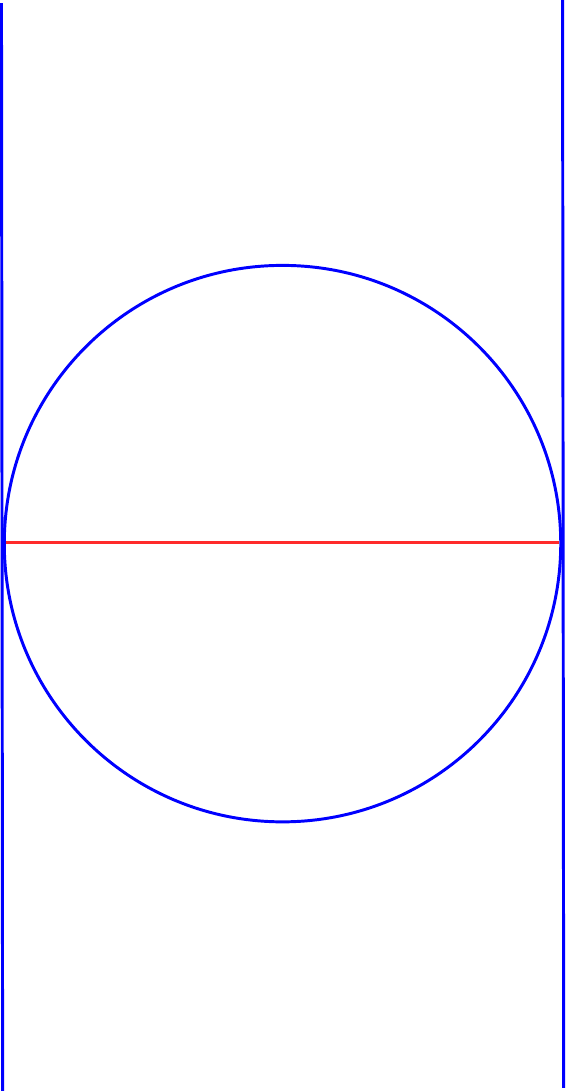}
				\caption{}
				\label{explain2}
			\end{subfigure}%
			\caption{}
		\end{figure}
		Hence $\tilde{S_{2}}$ intersects $\Delta$ at right angles and there is a right angle in the quadrangulation of $H$.
		This forces all angles in the quadrangulation to be right angles.
		Intersections of (topological) checkerboard surfaces at crossing arcs lift to intersections of totally geodesic planes $\tilde{S_{1}}$ and $\tilde{S_{2}}$ at geodesics.
	 	These geodesics connect centers of distinct horospheres which is lifts of $\partial N(K)$ in $\mathbb{H}^3$.
	 	The intersection angle can be read off from horospheres which is a right angle.
	 	Therefore two checkerboard surfaces intersect each other at right angles.
	 	
		If $K$ is a link.
		There is a 3-gon in $D$ by Lemma \ref{no bigon region} and Lemma \ref{3-gon exists}.
		Denote the collection of components of $K$ adjacent to this 3-gon as $\mathcal{C}$.
		Similar argument as above shows that for every component $C$ in $\mathcal{C}$, the quadrangulation on $\partial N(C)$ is in fact rectangulation.
		Because $L$ is hyperbolic, $L$ is non-split.
		An alternating link $L$ is not split iff its alternating diagram $D$ is not split(\cite{menasco1984closed},\cite[Theorem 4.2]{raymond1997introduction}).
		Hence every remaining component $C'$ of $L$ shares some crossing arc $\gamma$ with some component $C$ in $\mathcal{C}$.
		$\gamma$ lifts and is isotopic to a geodesic that connect two distinct horospheres which are lifts of $\partial N(C)$ and $\partial N(C')$ respectively.
		Then the rectangulation on the lift of $\partial N(C)$ forces a rectangulation on the lift of $\partial N(C')$.
		Thus quadrangulations on the boundaries of remaining horocusps are all rectangulations.
		Therefore two checkerboard surfaces intersect each other in right angle.
	\end{proof}
	\begin{figure} 
		\includegraphics[width=.6\linewidth]{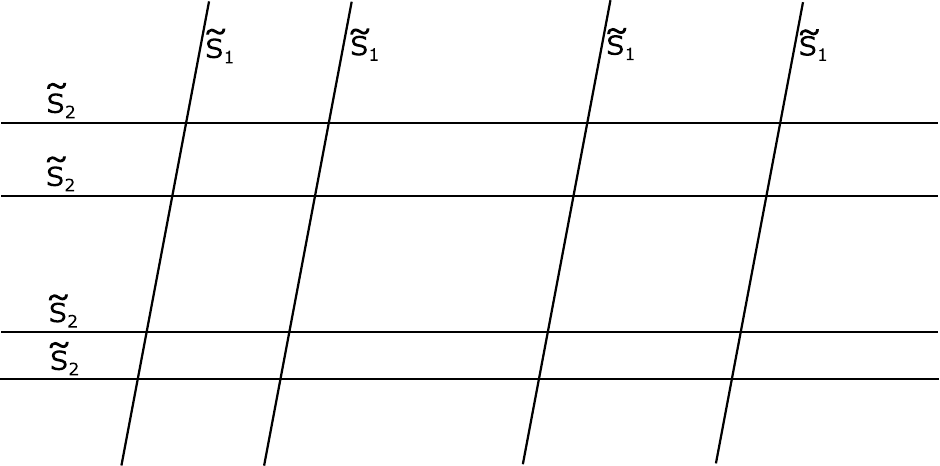}		
		\caption{quadrangulation on the horosphere}
		\label{quadragulation of horosphere}
	\end{figure}
	\begin{figure}
		\begin{subfigure}[t]{1\textwidth}
			\centering
			\includegraphics[width=.6\linewidth]{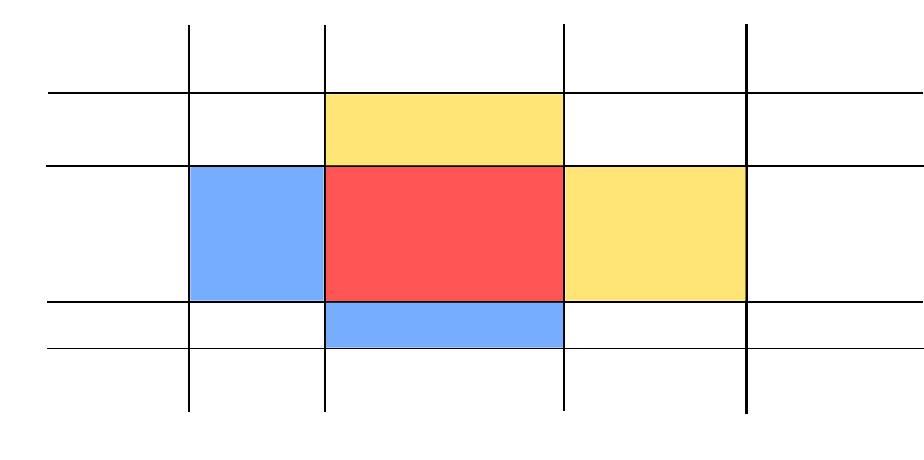}
			\caption{rectangulation on the horosphere}
			\label{cusp view}
		\end{subfigure}
		\begin{subfigure}[t]{1\textwidth}
			\centering
			\includegraphics[width=.6\linewidth]{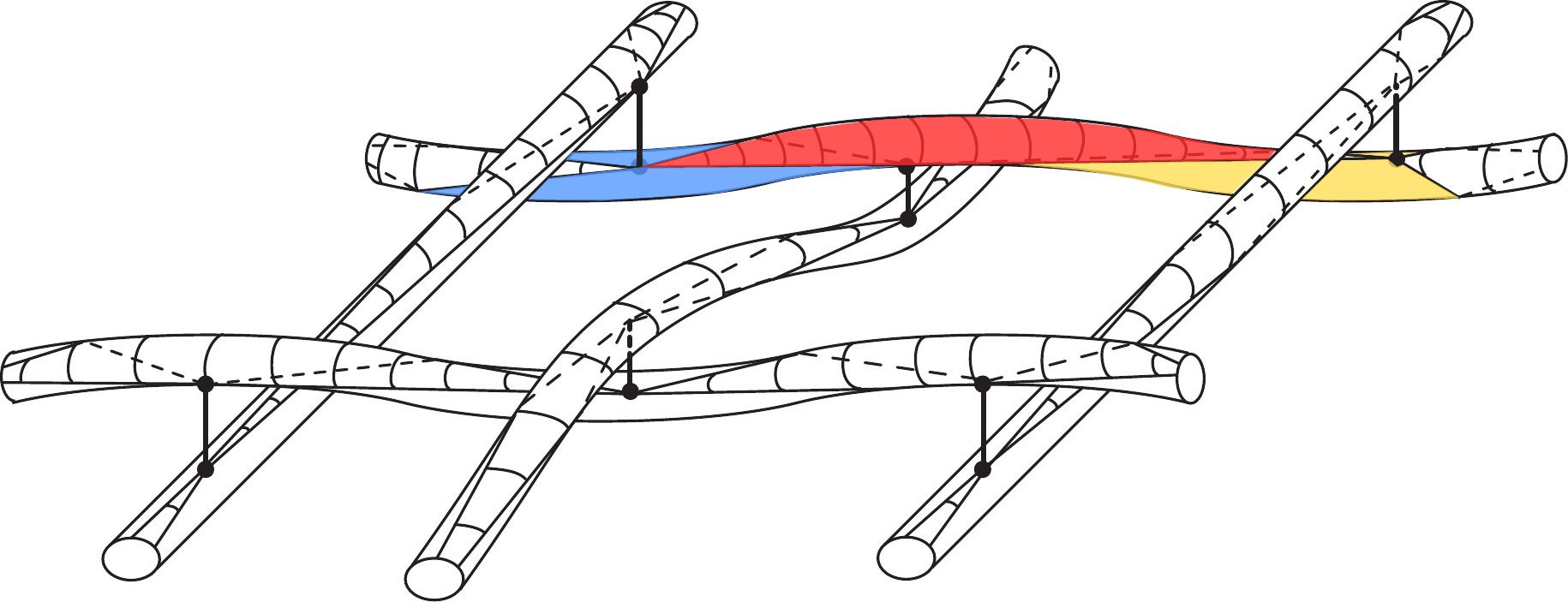}
			\caption{quadrangulation of $\partial N(K)$, figure modified from \cite{adams2006cusp}}
			\label{quadrangulation of knot neigh}
		\end{subfigure}
		\begin{subfigure}[t]{1\textwidth}
			\centering
			\includegraphics[width=.55\linewidth]{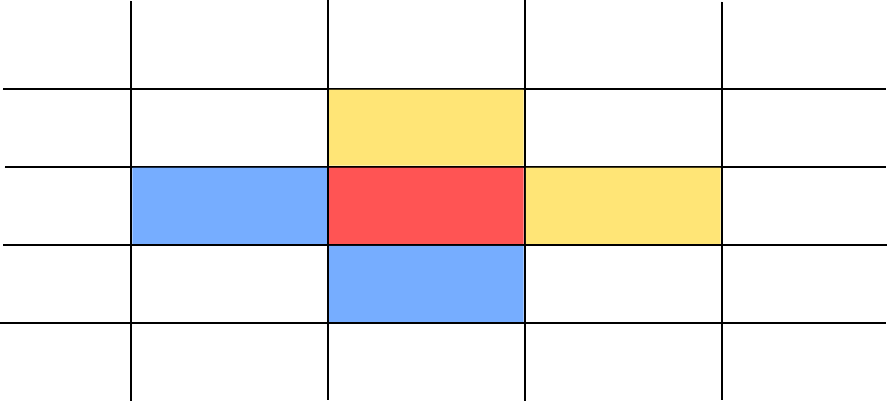}
			\caption{all rectangles are the same on the horosphere}
			\label{all rectangles same}
		\end{subfigure}%
		\caption{}
	\end{figure}
%
	\begin{Def}
		We call two checkerboard surfaces of the diagram of a knot or link $K$ \textit{black and white checkerboard surfaces}.
		Suppose they are totally geodesic in $S^3-K$, by the above proof, intersections of checkerboard surfaces and horocusps of $K$ give the boundaries of the horocups a rectangulation.
		We call edges of these rectangles coming from the black checkerboard surface \textit{horizontal segments}, and that coming from the white checkerboard surface \textit{vertical segments}.
	\end{Def}
	The proof of Theorem \ref{right angle} shows that two totally geodesic checkerboard surfaces give the boundaries of horocusps a rectangulation.
	Consider the rectangles on the boundary of a fixed horocusp $\partial N(C)$.
	One of the diagonals of these rectangles is the meridian of $\partial N(C)$ which is a geodesic on this Euclidean torus(see Figure \ref{cusp view} and Figure \ref{quadrangulation of knot neigh}).
	Hence these meridians all have equal length and all the diagonals of retangles have equal length.
	Since these rectangles share either a horizontal segment or a vertical segment, by Pythagoras theorem, they are all the same in the sense that after lifting them to the horosphere in the preimage of $\partial N(C)$, there is a Euclidean translation of the horosphere transforming one into another; see Figure $\ref{all rectangles same}$.
	Note rectangles in the boundaries of different horocusps might have different Euclidean structures up to similarity in general.
	As a result horizontal(vertical) segments in the same horocusp have equal length.

\subsection{A regularity condition}
	We define regular $n$-gon following \cite{aitchison2002archimedean}.
	\begin{Def} \label{def:regular}
		Let $F$ be a convex,planar,ideal $n$-gon in $\mathbb{H}^3$.
		We call $F$ \textit{regular} if $F$ is setwise invariant under some rotation of order n. 
	\end{Def}
	 Aitchison and Reeves prove the following lemma about regular $n$-gon\cite[Lemma 3.2]{aitchison2002archimedean}.
	\begin{Def}
		Let $p,q,r,s \in \mathbb{C}$, the \textit{cross ratio} of $p$,$q$,$r$ and $s$ is given by $R(p,q,r,s)=\frac{(p-r)(q-s)}{(p-s)(q-r)}$.
	\end{Def}
	\begin{Lemma}\label{4 points determine regular $n$-gon}
		F is regular iff for any four consecutive vertices $v_0$,$v_1$,$v_2$ and $v_3$ of $F$, the cross ratio of these points is given by $R(v_0,v_1,v_2,v_3)=1+\frac{1}{2cos(2\pi/n)+1}$.
	\end{Lemma}
	Here $F$ is sitting in the upper half space model of $\mathbb{H}^3$ and the ideal vertices of $F$ correspond to complex number in the complex plane.
	The cross ratio is calculated based on these complex number.
	Note the lemma works for 3-gons if we take $v_3=v_0$.
	As a result of the lemma, four consecutive vertices determine a regular $n$-gon uniquely including the integer $n$.
	\begin{Def}
		Let $K$ be a hyperbolic alternating knot or link with a reduced alternating diagram $D$.
		Suppose checkerboard surfaces of $D$ are totally geodesic in the complement of $K$.
		For every $n$-gon $G$ in $D$, $G$ is isotopic to some planar,ideal n-gon in a totally geodesic checkerboard surface $S$.
		In the universal cover, $S$ lifts to a set of geodesic planes and $G$ lifts to a set of planar,ideal n-gons in these planes.
		These n-gons relate to each other by covering transformations.
		We say that \textit{an $n$-gon $G$ in $D$ is regular} if any(hence all) $n$-gon in this set is regular in the sense of Definition \ref{def:regular}
	\end{Def}
	\begin{Lemma}\label{opposite regular are same $n$-gon}
		Let L be a hyperbolic alternating knot or link with a reduced alternating diagram $D$.
		Suppose that two checkerboard surfaces of D are totally geodesic in $S^3-L$.
		Let $G_n$,$G_m$($n,m\geq3$) be regular n-gon,m-gon in D respectively opposite to each other, then $n=m$.
	\end{Lemma}
	\begin{proof}
		$G_n$ and $G_m$ are in a common totally geodesic checkerboard surface, say the black checkerboard surface $S_b$.
		Lift them to $\mathbb{H}^3$, we have adjacent totally geodesic regular ideal $n$-gon and m-gon within a translate of the lifts of $S_b$(see Figure \ref{opposite gon same}).
		$QA$,$QA'$ are horizontal segments on the same horosphere, and $PD$,$PD'$ are horizontal segments on the same horosphere.
		By the discussion after Theorem \ref{right angle}, $QA=QA'$ and $PD=PD'$.
		As a result, points $E,\infty,O,F$ and $E',\infty,O,F'$ are in symmetric position.
		Note if $n$ or $m = 3$, then $F=E$ or $F'=E'$.
		After applying some isometry fixing $\infty$ in $\mathbb{H}$, one can assume that $E=z_1, E'=-z_1, O=0, F=z_2$ and $ F'=-z_2$ where $z_1$ and $z_2$ are some complex numbers.
		Denote the cross ratios of four consecutive vertices of $G_n$ and $G_m$ by $R(E,\infty,O,F)$ and $R(E',\infty,O,F')$ respectively.
		We have $R(E,\infty,O,F) = R(z_1,\infty,0,z_2) = R(-z_1,\infty,0,-z_2) = R(E', \infty,O,F')$.
		Since $G_n$ and $G_m$ are regular, by Lemma \ref{4 points determine regular $n$-gon} $n=m$.
	\end{proof}
	\begin{figure} 
		\includegraphics[width=.8\linewidth]{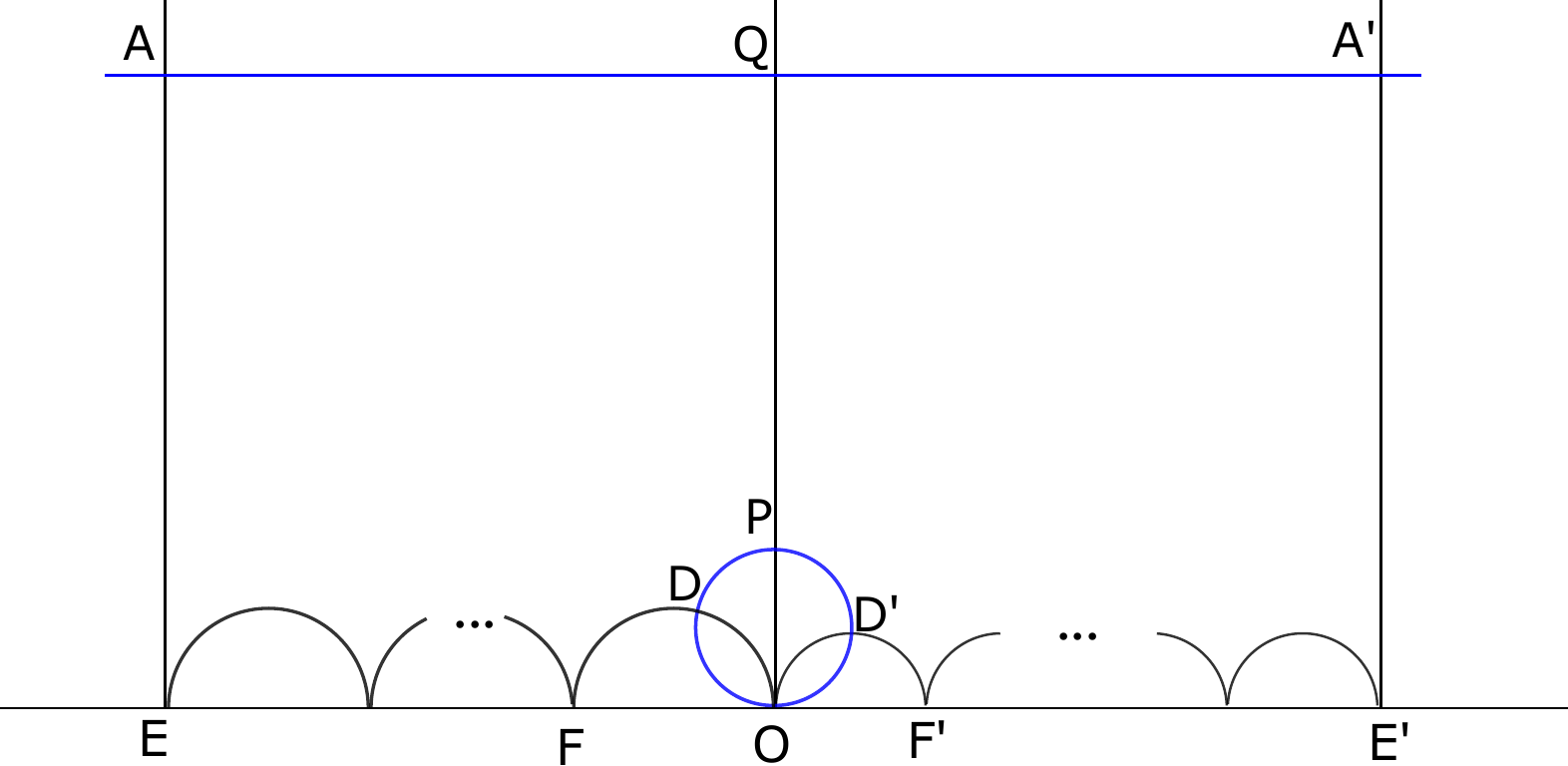}
		\caption{$n$-gon and m-gon in a totally geodesic plane. The m-gon and n-gon appear in black while the horospheres appear in blue.}
		\label{opposite gon same}
	\end{figure}
	\begin{Rem} \label{regular $n$-gon opposite to 3-gon is 3-gon}
		Since ideal 3-gon is regular, if there is a 3-gon in $D$ of the above lemma, then any regular $n$-gon in $D$ opposite to this 3-gon is a 3-gon.
	\end{Rem}
	\begin{Prop} \label{regular imply no}
		Let L be a hyperbolic alternating knot or link with a reduced alternating diagram $D$.
		If two checkerboard surfaces of D are totally geodesic in $S^3-L$ and every $n$-gon in D is regular, then L is one of the links in Figure \ref{threediagrams}.
	\end{Prop}
	\begin{proof}
		In the following, we consider knots as links with one component.
		Since every $n$-gon in $D$ is regular, by Lemma \ref{opposite regular are same $n$-gon}, opposite $n$-gons have the same number of sides.
		Therefore each checkerboard surface consists of $n$-gons that have the same number of sides.
		
		Since $D$ is reduced, there is no 1-gons.
		Since we assume that two checkerboard surfaces of $D$ are totally geodesic in the knot complement, by Lemma \ref{no bigon region}, there is no 2-gons in $D$.
		By Lemma \ref{3-gon exists}, $D$ has a 3-gon.
		Hence one of the checkerboard surfaces consists of 3-gons.
		We only need to decide what $n$-gon the other checkerboard surface consist of.
		Since there is no 1-gons or 2-gons, $n \geq 3$.
		Denote the projection graph by $G$.
		Let $V$,$E$ and $F$ be the number of vertices,edges and faces of $G$ respectively.
		Since the graph is 4-valent, $4V/2=E$. Let $F_3$ and $F_n$ be the number of 3-gons and $n$-gons respectively, then $F_3 +F_n =F$.
		By the above argument, $F_3=2V/3$ and $F_n=2V/n$. Finally by Euler's formula $V-E+F=2$. We have
		$$ 2 = V-E+F = V-E+F_3+F_n = V-2V+ \frac{2V}{3} + \frac{2V}{n} = V(\frac{2}{n}-\frac{1}{3}) > 0 $$
		Therefore $3 \leq n \leq 5$.
		
		If $n=3$, $G$ is the octahedron graph.
		If $n=4$, $G$ is the cuboctahedron graph.
		If $n=5$, $G$ is the icosidodecahedron graph.
		Turning these graphs into alternating diagrams obtains links in Figure \ref{threediagrams}.
		Note either way to alternate the graphs results the same link by a symmetry for each diagram.
	\end{proof}
	The regularity condition in the above theorem can be dropped.
	See Section \ref{removeRegularity}.
	\begin{Rem}\label{three links totally geodesic}
		Thurston describes the complete hyperbolic structure of the Borromean ring in \cite{thurston1979geometry}.
		Hatcher describes the checkerboard decomposition of the link in Figure \ref{cuboct} and proves the links in Figure \ref{octahedron} and Figure \ref{cuboct} are arithmetic\cite{hatcher1983hyperbolic}.
		Adams notices that links in Figure $\ref{threediagrams}$ have both checkerboard surfaces totally geodesic and two checkerboard surfaces intersect each other at right angles\cite[Example 1.12]{adams2007noncompact}.
		Also see \cite[Example 2.5]{adams2008totally} for a proof that both checkerboard surfaces of these links are totally geodesic.
	\end{Rem}
\begin{figure} 
	\begin{subfigure}[t]{.3\textwidth}
		\centering
		\includegraphics[width=.8\linewidth]{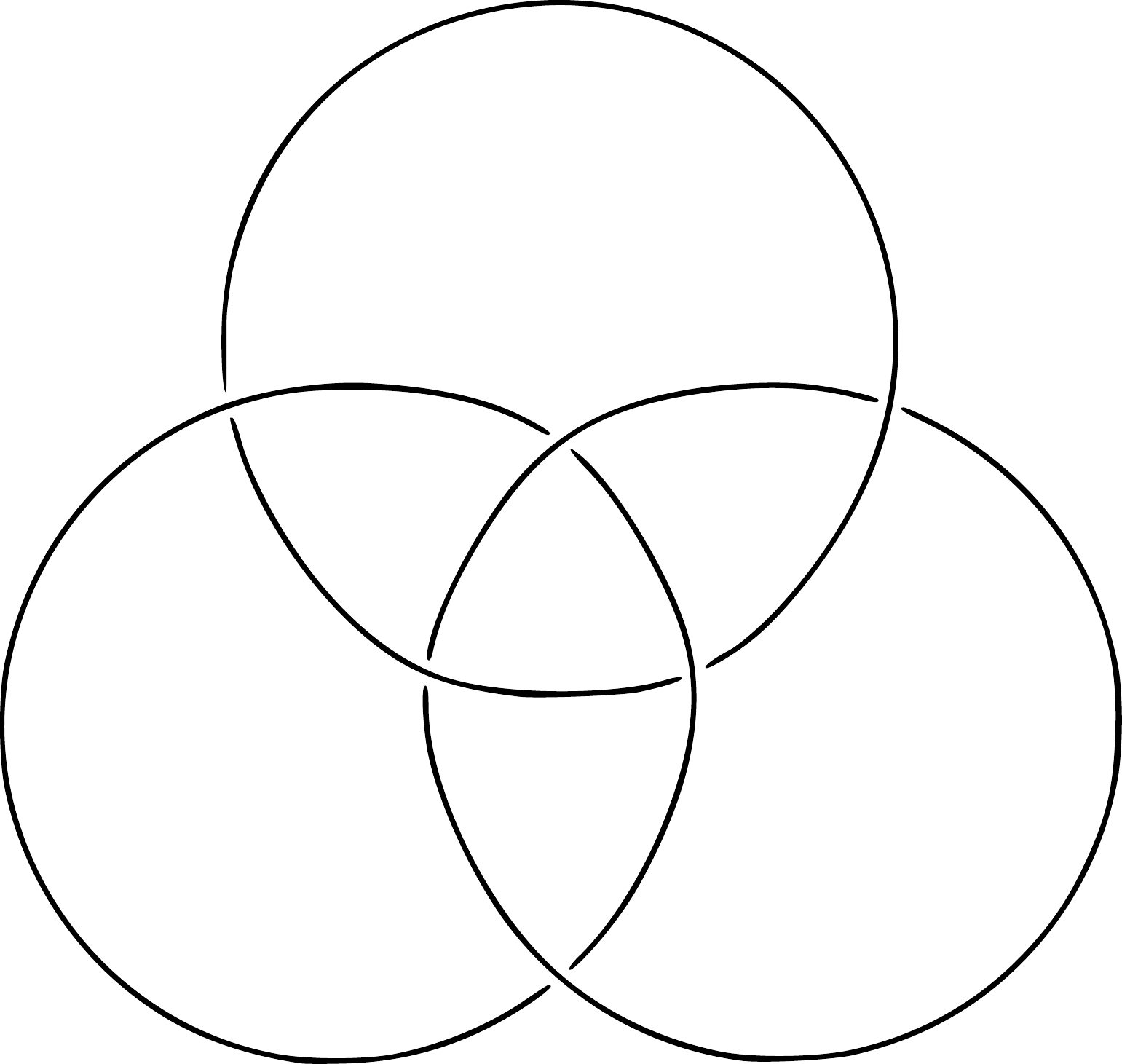}
		\caption{}
		\label{octahedron}
	\end{subfigure}
	\begin{subfigure}[t]{.3\textwidth}
		\centering
		\includegraphics[width=.8\linewidth]{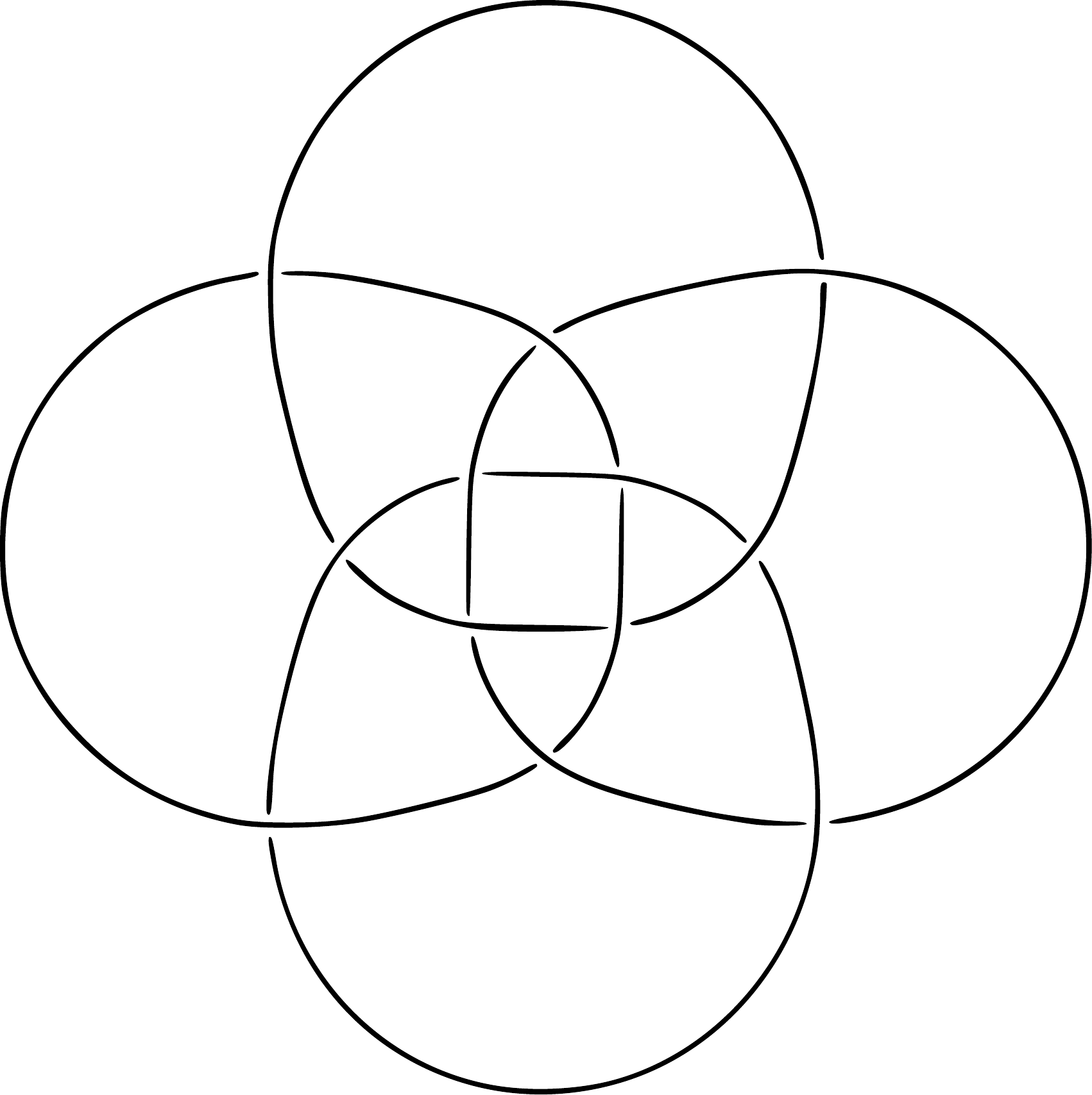}
		\caption{}
		\label{cuboct}
	\end{subfigure}
	\begin{subfigure}[t]{.3\textwidth}
		\centering
		\includegraphics[width=.8\linewidth]{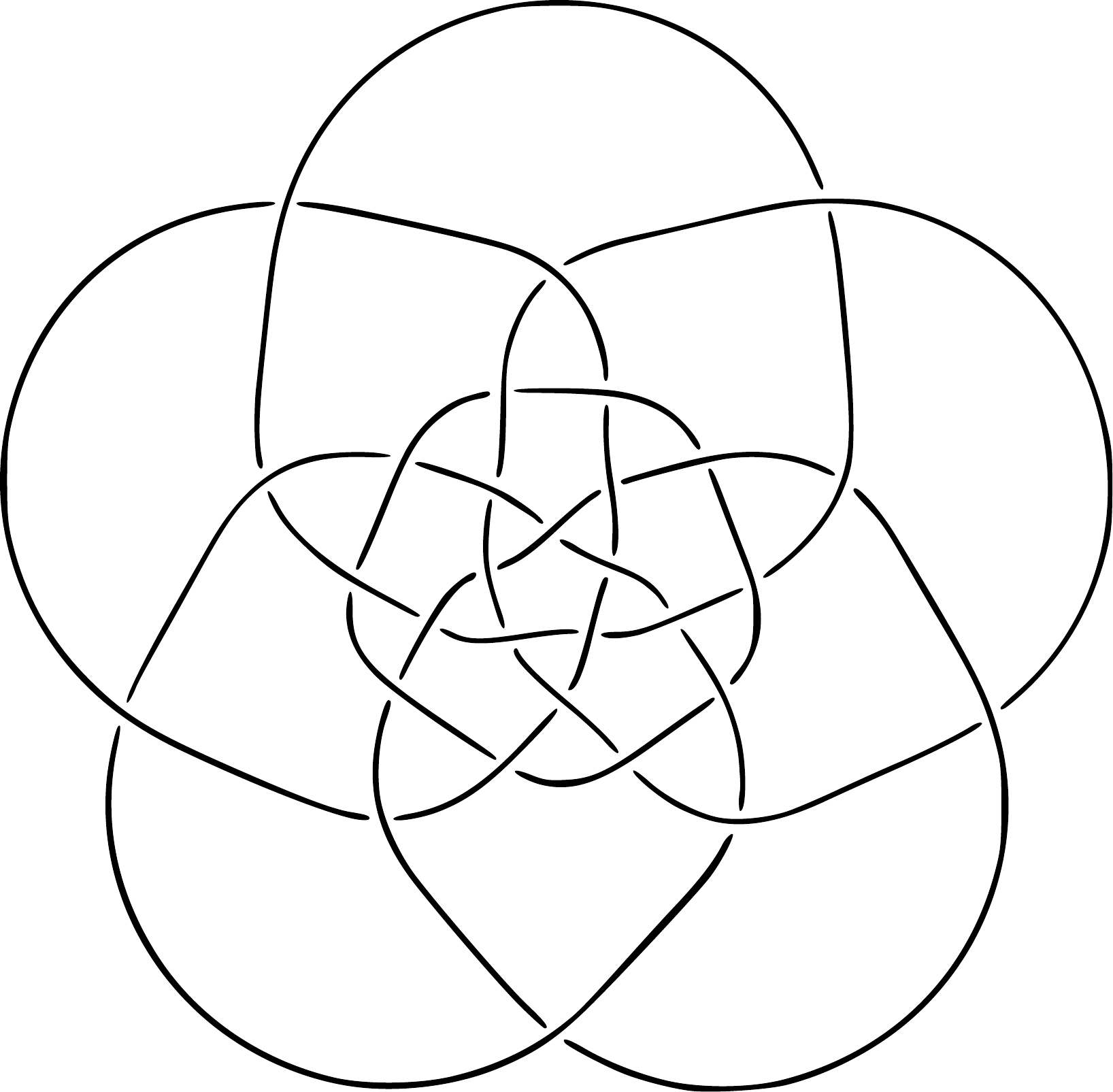}
		\caption{}
		\label{icosido}
	\end{subfigure}
	\caption{figures modified from Knotilus\cite{flint2006knotilus}}
	\label{threediagrams}
\end{figure}
\subsection{Right-angled completely realisable link}
	\begin{Def}
		We call an hyperbolic alternating link $L$ given by its reduced alternating diagram $D$ \textit{completely realisable} if two checkerboard polyhedra of $L$ associated to $D$ can be realised directly as ideal hyperbolic polyhedra and they can be glued together to give the complete hyperbolic structure of $S^3-L$.
		If in addition the polyhedra are right-angled, we call the link \textit{right-angled completely realisable}.
		Note right-angled ideal hyperbolic polyhedron is necessarily 4-valent.
		An ideal hyperbolic polyhedron is called \textit{regular-faced} if every face is regular.
		We say a combinatorial polyhedron $P$ \textit{regular faced realisable} if there is a regular-faced ideal hyperbolic polyhedron having the same combinatorial structure as $P$.
		If in addition all the face normals in the realisation intersect at some point, we call it \textit{simultaneously regular faced realisable} .
	\end{Def}
	\begin{Th} \label{completely realisable right-angled link}
		An alternating knot or link $L$ given by its reduced alternating diagram $D$ is right-angled completely realisable iff it is one of the links in Figure \ref{threediagrams}.
		In particular, there is no right-angled completely realisable knot.
	\end{Th}
	\begin{proof}
		If the alternating link $L$ is right-angled completely realisable, the link complement $S^3-L$ can be obtained by gluing two right-angled ideal hyperbolic polyhedra $\mathcal{P}^{+}$ and $\mathcal{P}^{-}$.
		Such polyhedron is the intersection of half spaces of $\mathbb{H}^3$, hence $\mathcal{P}^{+}$ and $\mathcal{P}^{-}$ are convex.
		The underlying combinatorial polyhedra of $\mathcal{P}^{+}$ and $\mathcal{P}^{-}$ are mirror image of each other.
		Rivin shows that a convex ideal hyperbolic polyhedron is uniquely determined by its dihedral angles\cite[Theorem 14.1]{rivin1994euclidean}.
		Since the corresponding edges of $\mathcal{P}^{+}$ and $\mathcal{P}^{-}$ have the same dihedral angle $\pi/2$, $\mathcal{P}^{+}$ and $\mathcal{P}^{-}$ are isometric by a reflection.
		Therefore corresponding faces of $\mathcal{P}^{+}$ and $\mathcal{P}^{-}$ are isometric ideal polygons which is glued by a "gear rotation".
		The gluing rotates the faces by one edge, hence all faces of $\mathcal{P}^{+}$ and $\mathcal{P}^{-}$ are regular and every $n$-gon in $D$ is regular.
		Each checkerboard surface of $D$ is obtained by attaching totally geodesic polygonal faces of the polyhedra along edges of the polyhedra.
		This gluing gives the checkerboard surface a pleating.
		Since the polyhedra are right-angled, the pleating angle at each edge is $\pi$.
		Hence there is no bending and the resultant surface is totally geodesic.
		It follows that two checkerboard surfaces of $D$ are totally geodesic in $S^3-L$.
		By Proposition \ref{regular imply no}, $L$ is one of the links in Figure \ref{threediagrams}.
		
		If $L$ is one of the links in Figure \ref{threediagrams}, let $P$ be its checkerboard polyhedron of the diagram.
		$P$ is a combinatorial octahedron,cuboctahedron or icosidodecahedron.
		Euclidean octahedron,cuboctahedron and icosidodecahedron are simultaneously regular faced and inscribable in the unit sphere.
		By considering the unit ball as the Klein model, $P$ is simultaneously regular faced realisable.
		Denote this realisation by $\mathcal{P}$.
		At each ideal vertex $v$ of $\mathcal{P}$, there is a (sufficiently small) horosphere $H$ intersects $\mathcal{P}$ in a quadrilateral.
		There are two reflection symmetries of $\mathcal{P}$ fixing $v$ and $H$.
		The reflection plane of one symmetry transverses two non-adjacent faces incident at $v$ and the reflection plane of the other symmetry transverses the other two non-adjacent faces incident at $v$.
		Hence two reflections preserve the quadrilateral and the quadrilateral is a rectangle.
		It follows that every dihedral angle of $\mathcal{P}$ is $\pi/2$.
		Glue $\mathcal{P}$ and its mirror image as the checkerboard decomposition, we obtain a hyperbolic structure on $S^3-L$.
		By Corollary 5.3 in \cite{aitchison2002archimedean}, this hyperbolic structure is complete.
		Hence $L$ is right-angled completely realisable.
	\end{proof}
	\begin{Rem} \label{resolveCKP}
		By Theorem 5.4 in \cite{champanerkar2019right}, if $L$ is one of the above three links, right-angled volume of $L$ is equal to twice of the volume of its checkerboard polyhedron with a right-angled ideal hyperbolic structure.
		Hence right-angled volume of $L$ is equal to its hyperbolic volume by the above theorem, which gives two more examples to the question of Champanerkar, Kofman and Purcell \textit{Does there exist a hyperbolic alternating link $L$, besides the Borromean link, for which $vol^{\perp}(L) = vol(L)$?}\cite{champanerkar2019right}.
		
		Champanerkar, Kofman and Purcell conjecture that there does not exist a right-angled knot\cite{champanerkar2019right}.
		That is there is no hyperbolic knot $K$ of which the complement $S^3-K$ with the complete hyperbolic structure admits a decomposition into ideal hyperbolic right-angled polyhedra.
		Theorem \ref{completely realisable right-angled link} resolves a case of the conjectue.
	\end{Rem}

\subsection{Removal of the regularity condition} \label{removeRegularity}
	Let $L$ be a hyperbolic alternating knot or link given by a reduced alternating diagram $D$.
	Suppose two checkerboard surfaces of $D$ are totally geodesic in $S^3-L$.
	By cutting along two totally geodesic checkerboard surfaces we obtain two checkerboard polyhedra $P^{+}$ and $P^{-}$ which are topological balls.
	They lift to embed in $\mathbb{H}^3$ as 3-balls.
	Consider a fundamental domain of these balls which are two 3-balls $\mathcal{P}^{+}$ and $\mathcal{P}^{-}$.
	Crossing arcs lift and are isotopic to geodesics in $\mathbb{H}^{3}$.
	None of the faces of checkerboard polyhedra degenerates while lifting by Proposition 2.1 in \cite{thistlethwaite2014alternative}, that is n arcs on $L$ of an n-gon in $D$ lifts to n distinct vertices in $\mathbb{H}^3$.
	Hence by assumptions the faces lift to geodesic ideal n-gons which are faces of $\mathcal{P}^{+}$ and $\mathcal{P}^{-}$.
	The dihedral angles of $\mathcal{P}^{+}$ and $\mathcal{P}^{-}$ are right angles by Theorem \ref{right angle}.
	$\mathcal{P}^{+}$ and $\mathcal{P}^{-}$ are convex since otherwise some geodesic plane which contains some face of  $\mathcal{P}^{+}$(respectively $\mathcal{P}^{-}$) would cut through $\mathcal{P}^{+}$(respectively $\mathcal{P}^{-}$), contradicting the fact that the totally geodesic checkerboard surfaces are already cut away.
	The balls $\mathcal{P}^{+}$ and $\mathcal{P}^{-}$ remain homeomorphic to a ball under the isotopy making faces geodesic, unless two distinct topological edges on the ball are isotopic in $\mathbb{H}^{3}$.
	But this would imply that the complement of the checkerboard surface had a bigon, contradicting Theorem 3.1 in \cite{adams2008totally}.
	It follows that  $\mathcal{P}^{+}$ and $\mathcal{P}^{-}$ are convex right-angled polyhedra with the same combinatorial structure as the checkerboard polyhedra.
	The gluing is determined by the topology which is the same as the checkerboard decomposition.
	Hence $L$ is right-angled completely realisable.
	By Theorem \ref{completely realisable right-angled link}, $L$ is one of the links in Figure \ref{threediagrams}.
	We have
	\begin{Th} \label{main main results}
		Let $L$ be a hyperbolic alternating knot or link given by a reduced alternating diagram $D$.
		If two checkerboard surfaces of D are totally geodesic in $S^3-L$, then L is one of the links in Figure \ref{threediagrams}.
		In particular, there is no alternating knot with two totally geodesic checkerboard surfaces.
	\end{Th}
		
\subsection{Weaving knot} \label{weaving knot}
	A \textit{weaving knot} $W(p,q)$ is the alternating knot or link with the same projection as the standard p-braid $(\sigma_1...\sigma_{(p-1)})^q$ projection of the torus knot or link $T(p,q)$.
	We require $p \geq 3$ and $q \geq 2$ in this definition and only consider hyperbolic weaving knots below.
	The diagram having the projection in the definition is called the standard diagram of a weaving knot.
	Denote it by $D$.
	\begin{Cor}
		All hyperbolic weaving knots $W(p,q)(p\geq3,q\geq2)$ except $W(3,3)$ and $W(4,4)$ have both checkerboard surfaces of the standard diagram not totally geodesic.
	\end{Cor}
	This result generalizes Example 1.13 in \cite{adams2007noncompact}.
	Note there is a symmetry between two checkerboard surfaces of the standard diagram of $W(p,q)$.
	As a consequence of Mostow-Prasad rigidity it is an isometry between these two surfaces.
	Therefore they are both totally geodesic or both not totally geodesic.
	\begin{proof}
		The crossing number of weaving knot $W(p,q)$ is $(p-1)*q$.
		By Theorem \ref{main main results}, we only need to examine the cases where $(p-1)*q=6, 12, 30$.
		One can check directly that only $W(3,3)$ and $W(4,4)$ are links in Figure \ref{threediagrams}.
	\end{proof}
	\begin{Rem}
		The \textit{infinite weave} $\mathcal{W}$ is defined to be the infinite alternating link with the square grid projection.
		Champankar, Kofman and Purcell prove that as $p,q \to \infty$,$S^3-W(p,q)$ approaches $\mathbb{R}^3-\mathcal{W}$ as a geometric limit\cite[Theorem 1.2]{champanerkar2016volume} and the link $\mathcal{W}$ have totally geodesic checkerboard surfaces\cite[Theorem 5.1]{champanerkar2019geometry}.
	\end{Rem}

\bibliographystyle{plain}

\end{document}